\documentclass[11pt]{amsart}
\usepackage{amsthm}
\newtheorem{de}{Definition}[section]

\newtheorem{lm}[de]{Lemma}
\newtheorem{pr}[de]{Proposition}

\newtheorem{te}[de]{Theorem}

\begin{document}
\title[Random Polymer]
{Diffusivity of Rescaled Random Polymer in Random Environment in dimensions $1$ and $2$}
\author {Zi Sheng Feng}
\address{zisheng.feng@utoronto.ca\linebreak
Department of Mathematics\\
University of Toronto\\
Toronto, Ontario, Canada, M5S 2E4}

\date{}
\begin{abstract} We show random polymer is diffusive in dimensions $1$ and $2$ in probability in an intermediate scaling regime. The scale is $\beta=
o(N^{-1/4})$ in $d=1$ and $\beta=o((\log N)^{-1/2})$ in $d=2$ as $N\rightarrow \infty$.
\end{abstract}
\maketitle

\section{Introduction}
Consider walks $\omega: [0,N]\bigcap \mathbb{Z}\rightarrow \mathbb{Z}^d$ such that $\omega(0)=0$, $|\omega(n)-\omega(n-1)|=1$. Let $P^N_0$ be uniform measure on the space of these walks each with weight $(2d)^{-N}$, then $$p_0(N,x):=P^N_0(\omega(N)=x)=\int 1_{[\omega(N)=x]}dP^N_0(\omega)=\frac{1}{(2d)^N}\sum_{\omega:\ \omega(N)=x}$$ is probability of the nearest neighbor simple random walk starting at $0$ is at site $x$ at time $N$.

Let the random environment be given by $h=\{h(n,x): n\in \mathbb{N}, x\in \mathbb{Z}^d\}$, a sequence of independent identically distributed random variables with $h(n,x)=\pm 1$ with equality probability on some probability space $(H,\mathcal{G}, Q)$, which are also independent of the simple random walk. We denote expectation over the environment space by $E_Q$.

We define the (unnormalized) polymer density by \begin{eqnarray*}p(N,x)=\int 1_{[\omega(N)=x]}\prod_{1\leq n\leq N}\left[1+c_{N,d}h(n,\omega(n))\right]dP^N_0(\omega)\end{eqnarray*}  where $c_{N,d}$\footnote{For example, we may take $c_{N,1}=N^{-(1/4+\epsilon)}$ and $c_{N,2}=\log N^{-(1/2+\epsilon)}$ for any $\epsilon>0$. Also the scale $\beta=o(N^{-1/4})$ for $d=1$ is first identified in \cite{AKQ}} is such that \begin{eqnarray}\label{SLG}\lim_{N\rightarrow \infty}c^2_{N,1}N^{1/2}=0\ {\mbox{for}}\ d=1;\ \lim_{N\rightarrow\infty}c^2_{N,2}\log N=0\ {\mbox{for}}\ d=2\end{eqnarray}

Since the polymer density is not normalized, to obtain the probability of the polymer at time $N$ is at site $x$, we define $$p_N(N,x)=p(N,x)/Z(N)$$ where $Z(N)$ is the partition function $$Z(N)=\sum_xp(N,x)=\int \prod_{1\leq n\leq N}\left[1+c_{N,d}h(n,\omega(n))\right]dP^N_0(\omega)$$

In this paper, we show the mean square displacement of the polymer when scaled by $N$ converges to $1$ in probability in both $d=1,2$. Precisely, let $\langle \omega(N)^2\rangle_{N,h}=\sum_x x^2p_N(N,x)$,
\begin{te}\label{MTHM} With rescaling of the polymer density by $c_{N,d}$, for $d=1,2$, $$\frac{\langle \omega(N)^2\rangle_{N,h}}{N}\rightarrow 1$$ in probability as $N\rightarrow \infty$. \end{te}

We note that $\langle \omega(N)^2\rangle_{N,h}=\frac{K(N)}{Z(N)}$, where $K(N)=\int \prod_{1\leq n\leq N}\left[1+\beta_{N,d}h(n,\omega(n))\right]\omega(N)^2dP^N_0(\omega)$. To show the result, we are going to estimate second moment of the top and bottom quantity, and find that

\begin{pr}\label{d12M} For $d=1$, $$i)\ E_Q(Z(N)^2)\leq \sum^N_{n=0}\left(c_1 c^2_{N,1}N^{1/2}\right)^n;\ \ \ ii)\ E_Q(K(N)^2)\leq N^2\sum^N_{n=0}\left(c_1 c^2_{N,1}N^{1/2}\right)^n$$  for some constant $c_1$ that depends only on the dimension. \end{pr}

\begin{pr}\label{d22M} For $d=2$, $$i)\ E_Q(Z(N)^2)\leq \sum^N_{n=0}\left(c_2 c^2_{N,2}\log N\right)^n;\ \ \ ii)\ E_Q(K(N)^2)\leq N^2\sum^N_{n=0}\left(c_2 c^2_{N,2}\log N\right)^n$$ for some constants $c_2$ that depends only on the dimension. \end{pr}

The paper is organized as follows. In section 2, we write out second moments of the top and bottom quantity in the mean square displacement of the polymer. In Section 3, we show Proposition \ref{d12M} and Theorem \ref{MTHM} for dimension $d=1$. In Section 4, we show Proposition \ref{d22M} and Theorem \ref{MTHM} for dimension $d=2$. In Section 5, we show some other results.

\section{Second Moment Expansions}
In this section, we are going to write out the second moments of the top and bottom quantity in the mean square displacement of the polymer.

\begin{lm}\label{2} \begin{eqnarray*}E_Q(Z^2(N))=\sum^N_{n=0}\sum_{1\leq i_1< \cdots <i_n\leq N}c^{2n}_{N,d}\sum_{x_1,\ldots, x_n}\prod^n_{k=1}p^2_0(i_k-i_{k-1}, x_k-x_{k-1})\end{eqnarray*} \end{lm}

\begin{proof} By definition, $Z_N=\int \prod_{1\leq n\leq N}\left[1+c_{N,d}h(n,\omega(n))\right]dP^N_0(\omega)$. Upon expanding, $$Z_N=\int \sum^N_{n=0}\sum_{1\leq i_1<\cdots <i_n\leq N}c^n_{N,d}\prod^n_{k=1}h(i_k, \omega(i_k))dP^N_0(\omega)$$ Let $f_n(\omega)=\sum_{1\leq i_1<\cdots <i_n\leq N}c^n_{N,d}\prod^n_{k=1}h(i_k, \omega(i_k))$ and $g_n=\int f_n(\omega)dP^N_0(\omega)$, we see $$Z_N^2=(g_0+g_1+\cdots +g_N)^2=\sum_{0\leq n,m\leq N}g_{n}g_{m}$$

For $n\neq m$, we have \begin{eqnarray*}E_Qg_ng_m&=&E_Q\int\sum_{1\leq i_1<\cdots<i_n\leq N}c^n_{N,d}\prod^n_{k=1}h(i_k,\omega(i_k))dP^N_0(\omega)\int\sum_{1\leq i'_1<\cdots<i'_m\leq N}c^m_{N,d}\prod^m_{l=1}h(i'_l,\omega'(i'_l))dP^N_0(\omega')\end{eqnarray*}
Note that if there is some $i_k$ that is different from all other $i'_l$'s (or vice versa), then by independence of the $h(n,x)$'s and that they have mean $0$, we have $$E_Q\prod^n_{k=1}\prod^m_{l=1}h(i_k,\omega(i_k)) h(i'_l,\omega'(i'_l))=0$$
By Fubini, $E_Qg_ng_m=0$. But since $n\neq m$, the $i_k$'s and $i'_l$'s cannot all be matched in pairs, so there must be some $i_k$ different from all other $i'_l$'s (or vice versa).

On the other hand, for $n=m$, we have \begin{eqnarray*}E_Qg^2_n&=&E_Q\int\sum_{1\leq i_1<\cdots<i_n\leq N}c^n_{N,d}\prod^n_{k=1}h(i_k,\omega(i_k))dP^N_0(\omega)\int\sum_{1\leq i'_1<\cdots<i'_n\leq N}c^n_{N,d}\prod^n_{k=1}h(i'_k,\omega'(i'_k))dP^N_0(\omega')
\\&=&E_Q\int \int \sum_{1\leq i_1<\cdots <i_n\leq N}c^{2n}_{N,d}\prod^n_{k=1}h(i_k,\omega(i_k))h(i_k,\omega'(i_k))dP^N_0(\omega)dP^N_0(\omega')
\\&+&E_Q\int \int \sum_{i_l\neq i'_l\ \textrm{for some}\ l\ \in \{1,\ldots,n\}}c^{2n}_{N,d}\prod^n_{k=1}h(i_k,\omega(i_k))h(i'_k,\omega'(i'_k))dP^N_0(\omega)dP^N_0(\omega')
\\&=&E_Q\int \int \sum_{1\leq i_1<\cdots <i_n\leq N}c^{2n}_{N,d}\prod^n_{k=1}1_{\omega(i_k)=\omega'(i_k)}dP^N_0(\omega)dP^N_0(\omega')
\\&+&E_Q\int \int \sum_{i_l\neq i'_l\ \textrm{for some}\ l\ \in \{1,\ldots,n\}}c^{2n}_{N,d}\prod^n_{k=1}h(i_k,\omega(i_k))h(i'_k,\omega'(i'_k))dP^N_0(\omega)dP^N_0(\omega')
\\&=&\int\int \sum_{1\leq i_1<\cdots <i_n\leq N}c^{2n}_{N,d}\prod^n_{k=1}1_{\omega(i_k)=\omega'(i_k)}dP^N_0(\omega)dP^N_0(\omega')\end{eqnarray*} Third equality follows because $h^2(n,x)=1$ and nonzero contribution only comes from when all sites $\omega(i_k)$ and $\omega'(i_k)$ are matched in pairs. Fourth equality follows because if the $i_{\alpha}$'s were to match perfectly with the $i'_{\beta}$'s for $\alpha\neq \beta$, then we would get contradiction in terms of the order of the times. For example, take $n=3$ and the perfect cross matching $i_1=i'_2$, $i_2=i'_3$, $i_3=i'_1$, then by $i_1<i_2<i_3$ we would have $i'_2<i'_3<i'_1$, which is a contradiction.

Now, we are going to write out the integrals as sums in terms of the transition probabilities of the two independent walks. By above, we have $$E_Q(Z^2(N))=\sum^N_{n=0}\sum_{1\leq i_1< \cdots <i_n\leq N}c^{2n}_{N,d}\int\int 1_{\left[\omega(i_1)=\tilde{\omega}(i_1),\ldots, \omega(i_n)=\tilde{\omega}(i_n)\right]}dP^N_0(\omega)dP^N_0(\tilde{\omega})$$ $$=\sum^N_{n=0}\sum_{1\leq i_1< \cdots <i_n\leq N}c^{2n}_{N,d}\int \sum_{x_1,\ldots, x_n,x}1_{\left[\omega(i_1)=x_1,\ldots, \omega(i_n)=x_n\right]}$$ $$\times P^N_0\left(\tilde{\omega}(i_1)=x_1,\ldots,\tilde{\omega}(i_n)=x_n, \tilde{\omega}(N)=x\right)dP^N_0(\omega)$$ $$=\sum^N_{n=0}\sum_{1\leq i_1< \cdots <i_n\leq N}c^{2n}_{N,d}\int \sum_{x_1,\ldots, x_n}1_{\left[\omega(i_1)=x_1,\ldots, \omega(i_n)=x_n\right]}$$
 $$\times \prod^n_{k=1}p_0(i_k-i_{k-1}, x_k-x_{k-1})\sum_xp_0(N-i_n, x-x_n)dP^N_0(\omega)$$
where in the first equality we also need to sum over sites at time $N$ because $P^N_0$ is measure for walks of length $N$, and in the last equality we use the fact that increments of the walk are independent, and the walk is spatial homogeneous, i.e. probability of the walk starting at $y$ and ending at $x$ is same as probability of the walk starting at $0$ and ending at $y-x$. Next we note that $\sum_xp_0(N-i_n, x-x_n)=1$ because $p_0(n,x)$ is a transition probability.

Combining above and expanding similarly for the second walk we thus have shown Lemma \ref{2}.
\end{proof}

\begin{lm}\label{3}\begin{eqnarray*}E_Q(K^2(N))=\sum^N_{n=0}\sum_{1\leq i_1< \cdots <i_n\leq N}c^{2n}_{N,d}\sum_{x_1,\ldots, x_n} \prod^n_{k=1}p^2_0(i_k-i_{k-1}, x_k-x_{k-1})\end{eqnarray*} $$\times\left(\sum_x x^2p_0(N-i_n, x-x_n)\right)^2$$\end{lm}

\begin{proof} To estimate second moment of $K(N)$, we have $$E_Q(K^2(N))=E_Q\int\int \prod_{1\leq n\leq N}\left[1+c_{N,d}h(n,\omega(n))\right]\left[1+c_{N,d}h(n,\tilde{\omega}(n))\right]$$ $$\times\omega(N)^2\tilde{\omega}(N)^2dP^N_0(\omega)dP^N_0(\tilde{\omega})$$

As we see, the only difference between $E_Q(Z^2(N))$ and $E_Q(K^2(N))$ is the extra term $\omega(N)^2\tilde{\omega}(N)^2$, and we proceed as before to expand the second moment to get $$\sum^N_{n=0}\sum_{1\leq i_1< \cdots <i_n\leq N}c^{2n}_{N,d}\int\int 1_{\left[\omega(i_1)=\tilde{\omega}(i_1),\ldots, \omega(i_n)=\tilde{\omega}(i_n)\right]}\omega(N)^2\tilde{\omega}(N)^2dP^N_0(\omega)dP^N_0(\tilde{\omega})$$ $$=\sum^N_{n=0}\sum_{1\leq i_1< \cdots <i_n\leq N}c^{2n}_{N,d}\int \sum_{x_1,\ldots, x_n,x}1_{\left[\omega(i_1)=x_1,\ldots, \omega(i_n)=x_n\right]}$$
$$\times x^2P^N_0\left(\tilde{\omega}(i_1)=x_1,\ldots,\tilde{\omega}(i_n)=x_n, \tilde{\omega}(N)=x\right)\omega(N)^2dP^N_0(\omega)$$ $$=\sum^N_{n=0}\sum_{1\leq i_1< \cdots <i_n\leq N}c^{2n}_{N,d}\int \sum_{x_1,\ldots, x_n}1_{\left[\omega(i_1)=x_1,\ldots, \omega(i_n)=x_n\right]} $$ $$\times \prod^n_{k=1}p_0(i_k-i_{k-1}, x_k-x_{k-1})\sum_x x^2p_0(N-i_n, x-x_n)\omega(N)^2dP^N_0(\omega)$$ $$=\sum^N_{n=0}\sum_{1\leq i_1< \cdots <i_n\leq N}c^{2n}_{N,d}\sum_{x_1,\ldots, x_n} \prod^n_{k=1}p^2_0(i_k-i_{k-1}, x_k-x_{k-1})$$
$$\times \sum_x x^2p_0(N-i_n, x-x_n)\sum_y y^2p_0(N-i_n, y-x_n)$$
\end{proof}

\section{Diffusivity of Rescaled Random Polymer in $d=1$}
In this section, we are going to show Proposition \ref{d12M} and Theorem \ref{MTHM} for dimension $d=1$.

The key ingredient we need is that the transition probability $p_0(n,x)$ has the following estimate by the Gaussian density, more precisely, for $d\geq 1$, $x\in \mathbb{Z}^d$ such that $x_1+\cdots + x_d+n\equiv 0\mod 2$, then \begin{eqnarray}\label{GLE}p_0(n,x)=2\left(\frac{d}{2\pi n}\right)^{d/2}\exp\left(-\frac{d|x|^2}{2n}\right)+r_n(x)\end{eqnarray} where $|r_n(x)|\leq \min\left(c_dn^{-(d+2)/2}, c'_d|x|^{-2}n^{-d/2}\right)$ for some constants $c_d, c'_d$ that depend only on the dimension. (See Theorem 1.2.1 in \cite{Lawler})

\subsection*{Section 3.1}
In this subsection, we are going to show Proposition \ref{d12M} i) in a series of lemmas.

\begin{lm}\label{4} \begin{eqnarray*}\sum_{x_1,\ldots, x_n}\prod^n_{k=1}p^2_0(i_k-i_{k-1}, x_k-x_{k-1})\leq c_1^n\prod^n_{k=1}(i_k-i_{k-1})^{-1/2}\end{eqnarray*} for some constant $c_1$ that depends only on the dimension $d=1$ \end{lm}

\begin{proof} For $d=1$, since $e^{-x^2}\leq 1$ for all $x$, we see (\ref{GLE}) is at most $c_1n^{-1/2}$ for some constant $c_1$. Using this uniform estimate for each of the $p_0(i_k-i_{k-1},x_k-x_{k-1})$'s, we have $$\sum_{x_1,\ldots, x_n}\prod^n_{k=1}p^2_0(i_k-i_{k-1}, x_k-x_{k-1})=\sum_{x_1}p^2_0(i_1,x_1)\cdots \sum_{x_n}p^2_0(i_n-i_{n-1},x_n-x_{n-1})$$ $$\leq c_1^ni_1^{-1/2}\cdots (i_n-i_{n-1})^{-1/2}\sum_{x_1}p_0(i_1,x_1)\cdots \sum_{x_n}p_0(i_n-i_{n-1},x_n-x_{n-1})=c_1^n\prod^n_{k=1}(i_k-i_{k-1})^{-1/2}$$ for some constant $c_1$ that depends only on the dimension $d=1$ (last equality follows from that $p_0(n,x)$ is a transition probability). \end{proof}

\begin{lm}\label{5} \begin{eqnarray*}\sum_{1\leq i_1< \cdots <i_n\leq N}c_1^nc^{2n}_{N,1}\prod^n_{k=1}(i_k-i_{k-1})^{-1/2}\leq \left(c_1c^2_{N,1} N^{1/2}\right)^n\end{eqnarray*}
\end{lm}

\begin{proof}
$$\sum_{1\leq i_1< \cdots <i_n\leq N}c_1^nc^{2n}_{N,1}\prod^n_{k=1}(i_k-i_{k-1})^{-1/2} $$ $$=c_1^nc^{2n}_{N,1}\sum^{N-(n-1)}_{i_1=1}\cdots \sum^{N-1}_{i_{n-1}=i_{n-2}+1}i_1^{-1/2}\cdots (i_{n-1}-i_{n-2})^{-1/2} \sum^N_{i_n=i_{n-1}+1}(i_n-i_{n-1})^{-1/2}$$ $$\leq c_1^nc_{N,1}^{2n}\sum^{N-(n-1)}_{i_1=1}\cdots \sum^{N-1}_{i_{n-1}=i_{n-2}+1}i_1^{-1/2}\cdots (i_{n-1}-i_{n-2})^{-1/2} 2N^{1/2}$$ Last inequality holds because $\sum^{N}_{k=1}k^{-1/2}\leq 1+\int^N_1x^{-1/2}dx=1+2\left(N^{1/2}-1\right)\leq 2N^{1/2}$. Continuing from above and arguing similarly to estimate each sum in the expression we have
$$\leq c_1^nc^{2n}_{N,1}N^{1/2}\sum^{N-(n-1)}_{i_1=1}\cdots \sum^{N-1}_{i_{n-1}=i_{n-2}+1}i_1^{-1/2}\cdots (i_{n-1}-i_{n-2})^{-1/2}\leq \left(c_1c^{2}_{N,1}N^{1/2}\right)^n$$ where the constant $c_1$ will change from line to line (again it depends only on the dimension $d=1$).
\end{proof}

We conclude by Lemma \ref{2} that Proposition \ref{d12M} i) holds.

\subsection*{Section 3.2}
In this subsection, we are going to show Proposition \ref{d12M} ii) in a series of lemmas.

By standard computations of the moments of simple random walk of length $n$ in dimension $d=1$ using characteristic function, we have $$\sum_xx^2p_0(n,x)=n; \ \ \ \sum_xx^4p_0(n,x)=3n^2-2n$$

\begin{lm} \label{6}\begin{eqnarray*}\sum_x x^2p_0(N-i_n,x-x_n)=(N-i_n)+x^2_n\end{eqnarray*}\end{lm}

\begin{proof} $$\sum_x x^2p_0(N-i_n,x-x_n)
=\sum_{x-x_n}x^2p_0(N-i_n,x-x_n)=\sum_x(x+x_n)^2p_0(N-i_n,x)$$ $$=\sum_x
x^2p_0(N-i_n,x)+\sum_xx^2_np_0(N-i_n,x)=(N-i_n)+x^2_n$$ where first equality holds because summation over all $x$'s is same as summation over all $x-x_n$'s and third equality because holds any odd moment of simple random walk vanish.\end{proof}

\begin{lm} \label{7}\begin{eqnarray*}\sum_{x_k}x^4_kp_0(i_k-i_{k-1},x_k-x_{k-1})\end{eqnarray*} $$=3(i_k-i_{k-1})^2-2(i_k-i_{k-1})+6x^2_{k-1}(i_k-i_{k-1})+x^4_{k-1}$$ \end{lm}

\begin{proof} $$\sum_{x_k}x^4_kp_0(i_k-i_{k-1},x_k-x_{k-1})=\sum_{x_k-x_{k-1}}x^4_kp_0(i_k-i_{k-1},x_k-x_{k-1})
=\sum_{x_k}(x_k+x_{k-1})^4p_0(i_k-i_{k-1},x_k) $$ $$=\sum_{x_k}x^4_kp_0(i_k-i_{k-1},x_k)+\sum_{x_k}6x^2_kx^2_{k-1}p_0(i_k-i_{k-1},x_k)
+\sum_{x_k}x^4_{k-1}p_0(i_k-i_{k-1},x_k) $$ $$=3(i_k-i_{k-1})^2-2(i_k-i_{k-1})+6x^2_{k-1}(i_k-i_{k-1})+x^4_{k-1}$$ \end{proof}

\begin{lm} \label{8}\begin{eqnarray*}\sum_{x_1,\ldots, x_n}\prod^n_{k=1}p_0(i_k-i_{k-1}, x_k-x_{k-1})\left((N-i_n)^2+2(N-i_n)x^2_n+x^4_n\right) \end{eqnarray*}  $$=(N-i_n)^2+2(N-i_n)i_n+3\sum^n_{k=1}(i_k-i_{k-1})^2-2i_n+6\sum^n_{k=1}(i_k-i_{k-1})i_{k-1}$$ \end{lm}

\begin{proof}
We do this by induction on $n$. For $n=1$, we have $$\sum_{x_1}p_0(i_1,x_1)\left((N-i_1)^2+2(N-i_1)x^2_1+x^4_1\right) $$ $$=(N-i_1)^2\sum_{x_1}p_0(i_1,x_1)+2(N-i_1)\sum_{x_1}x^2_1p_0(i_1,x_1)
+\sum_{x_1}x^4_1p_0(i_1,x_1)$$ $$=(N-i_1)^2+2(N-i_1)i_1+3i_1^2-2i_1$$ (Note we do not have term of the form $6\sum^{n}_{k=1}(i_k-i_{k-1})i_{k-1}$ because $i_0=0$.)

Suppose equality holds for $n-1$. Then $$\sum_{x_1,\ldots, x_n}\prod^n_{k=1}p_0(i_k-i_{k-1}, x_k-x_{k-1})\left((N-i_n)^2+2(N-i_n)x^2_n+x^4_n\right)
$$ $$=\sum_{x_1\ldots, x_{n-1}}\prod^{n-1}_{k=1}p_0(i_k-i_{k-1}, x_k-x_{k-1})\sum_{x_n}p_0(i_n-i_{n-1}, x_n-x_{n-1})\left((N-i_n)^2+2(N-i_n)x^2_n+x^4_n\right)$$ $$=\sum_{x_1\ldots, x_{n-1}}\prod^{n-1}_{k=1}p_0(i_k-i_{k-1}, x_k-x_{k-1})
 [(N-i_n)^2+2(N-i_n)(i_n-i_{n-1}+x^2_{n-1})$$ $$+3(i_n-i_{n-1})^2-2(i_n-i_{n-1})+6x^2_{n-1}(i_n-i_{n-1})+x^4_{n-1}]$$
$$ =\sum_{x_1\ldots, x_{n-1}}\prod^{n-1}_{k=1}p_0(i_k-i_{k-1}, x_k-x_{k-1})([(N-i_n)^2+2(N-i_n)(i_n-i_{n-1})+3(i_n-i_{n-1})^2-2(i_n-i_{n-1})]$$
$$+[2(N-i_n)+6(i_n-i_{n-1})]x^2_{n-1}+x^4_{n-1})$$ $$=[(N-i_n)^2+2(N-i_n)(i_n-i_{n-1})+3(i_n-i_{n-1})^2-2(i_n-i_{n-1})]+
[2(N-i_n)+6(i_n-i_{n-1})]i_{n-1} $$ $$+3\sum^{n-1}_{k=1}(i_k-i_{k-1})^2-2i_{n-1}+6\sum^{n-1}_{k=1}(i_k-i_{k-1})i_{k-1}$$ $$=(N-i_n)^2+2(N-i_n)(i_n-i_{n-1}+i_{n-1})+3(i_n-i_{n-1})^2+3\sum^{n-1}_{k=1}(i_k-i_{k-1})^2$$ $$-2(i_n-i_{n-1}+i_{n-1})+6(i_n-i_{n-1})i_{n-1}+6\sum^{n-1}_{k=1}
(i_k-i_{k-1})i_{k-1}$$ $$=(N-i_n)^2+2(N-i_n)i_n+3\sum^n_{k=1}(i_k-i_{k-1})^2-2i_n+6\sum^n_{k=1}(i_k-i_{k-1})i_{k-1}$$ where in the second equality we use Lemmas \ref{6} and \ref{7} and in the fourth equality we use the inductive hypothesis.
\end{proof}

\begin{lm} \label{9}\begin{eqnarray*}(N-i_n)^2+2(N-i_n)i_n+3\sum^n_{k=1}(i_k-i_{k-1})^2 -2i_n+6\sum^n_{k=1}(i_k-i_{k-1})i_{k-1}\leq 100^nN^2\end{eqnarray*}\end{lm}

\begin{proof} $$(N-i_n)^2+2(N-i_n)i_n+3\sum^n_{k=1}(i_k-i_{k-1})^2-2i_n+6\sum^n_{k=1}(i_k-i_{k-1})i_{k-1}$$
 $$=N^2-2Ni_n+i^2_n+2Ni_n-2i^2_n+3\sum^n_{k=1}(i^2_k-2i_ki_{k-1}+i^2_{k-1})-2i_n+6\sum^n_{k=1}(i_ki_{k-1}-i^2_{k-1})$$ $$\leq N^2+2N^2+N^2+2N^2+2N^2+3n(N^2+2N^2+N^2)+2N^2+6n(N^2+N^2) $$ $$=10N^2+24nN^2\leq 100^nN^2$$ \end{proof}

\begin{lm} \label{10}\begin{eqnarray*}\sum_{x_1,\ldots, x_n} \prod^n_{k=1}p^2_0(i_k-i_{k-1}, x_k-x_{k-1})\left(\sum_x x^2p_0(N-i_n, x-x_n)\right)^2\leq c^n_1N^2\prod^n_{k=1}(i_k-i_{k-1})^{-1/2}\end{eqnarray*}\end{lm}

\begin{proof} Using the uniform estimate (\ref{GLE}) for $d=1$ on the transition probability as in the proof of Lemma \ref{4}, we get $$\sum_{x_1,\ldots, x_n} \prod^n_{k=1}p^2_0(i_k-i_{k-1}, x_k-x_{k-1})\left(\sum_x x^2p_0(N-i_n, x-x_n)\right)^2 $$ $$\leq c_1^ni_1^{-1/2}(i_2-i_1)^{-1/2}\cdots (i_n-i_{n-1})^{-1/2}\sum_{x_1,\ldots, x_n} \prod^n_{k=1}p_0(i_k-i_{k-1}, x_k-x_{k-1})\left((N-i_n)+x^2_{n}\right)^2$$ $$= c_1^ni_1^{-1/2}(i_2-i_1)^{-1/2}\cdots (i_n-i_{n-1})^{-1/2}\sum_{x_1,\ldots, x_n} \prod^n_{k=1}p_0(i_k-i_{k-1}, x_k-x_{k-1})$$ $$\times\left((N-i_n)^2+2(N-i_n)x^2_n+x^4_n\right)$$ $$=c_1^ni_1^{-1/2}(i_2-i_1)^{-1/2}\cdots (i_n-i_{n-1})^{-1/2} $$ $$\times
\left((N-i_n)^2+2(N-i_n)i_n+3\sum^n_{k=1}(i_k-i_{k-1})^2-2i_n+6\sum^n_{k=1}(i_k-i_{k-1})i_{k-1}\right)$$ $$\leq c_1^nN^2\prod^n_{k=1}(i_k-i_{k-1})^{-1/2}$$ where in the first inequality we also use Lemma \ref{6} to compute the second moment, in the second equality we use Lemma \ref{8} and in the last inequality we use Lemma \ref{9}.
\end{proof}

\begin{lm} \label{11}\begin{eqnarray*}\sum_{1\leq i_1< \cdots <i_n\leq N}c^n_1c^{2n}_{N,1}N^2\prod^n_{k=1}(i_k-i_{k-1})^{-1/2}\leq N^2\left(c_1c^2_{N,1}N^{1/2}\right)^n\end{eqnarray*}\end{lm}

\begin{proof} It follows from Lemma \ref{5}.
\end{proof}

We conclude by Lemma \ref{3} that Proposition \ref{d12M} ii) holds.

\subsection*{Section 3.3}
In this subsection, we are going to use Proposition \ref{d12M} to show Theorem \ref{MTHM} for $d=1$. We do so with a series of lemmas.

\begin{lm}\label{12}\begin{eqnarray*} E_Q\left((Z(N)-1)^2\right)\leq \sum^N_{n=1}\left(c_1c^2_{N,1}N^{1/2}\right)^n\end{eqnarray*}\end{lm}

\begin{proof}\begin{eqnarray*}E_Q\left((Z(N)-1)^2\right)&=&E_Q\left(Z^2(N)-2Z(N)+1\right)\\&=&E_Q(Z^2(N))-2E_Q(Z(N))+1\\&=&E_Q(Z^2(N))-2+1\\&=&E_Q(Z^2(N))-1 \\&=&\sum^N_{n=1}\sum_{1\leq i_1< \cdots <i_n\leq N}c^{2n}_{N,1}\sum_{x_1,\ldots, x_n}\prod^n_{k=1}p^2_0(i_k-i_{k-1}, x_k-x_{k-1})\\&\leq& \sum^N_{n=1}\left(c_1c^2_{N,1}N^{1/2}\right)^n\end{eqnarray*} where in the third equality we use $E_Q(Z(N))=1$ because the $h(n,x)$ have mean $0$, in the fifth equality we use that the $0$-th term in the second moment expansion of $Z(N)$ is $1$ (see Lemma \ref{2}) and the last inequality follows from Proposition \ref{d12M} i).\end{proof}

\begin{lm}\label{13}\begin{eqnarray*}E_Q\left((K(N)-N)^2\right)\leq N^2\sum^N_{n=1}\left(c_1c^2_{N,1}N^{1/2}\right)^n\end{eqnarray*}\end{lm}

\begin{proof}\begin{eqnarray*}E_Q\left((K(N)-N)^2\right)&=&E_Q(K^2(N)-2K(N)N+N^2)\\&=&E_Q(K^2(N))-2NE_Q(K(N))+N^2
\\&=&E_Q(K^2(N))-2N^2+N^2\\&=&
E_Q(K^2(N))-N^2\\&=&\sum^N_{n=1}\sum_{1\leq i_1< \cdots <i_n\leq N}c^{2n}_{N,1}\sum_{x_1,\ldots, x_n} \prod^n_{k=1}p^2_0(i_k-i_{k-1}, x_k-x_{k-1})\\&\times&\left(\sum_x x^2p_0(N-i_n, x-x_n)\right)^2\\&\leq& N^2\sum^N_{n=1}\left(c_1c^2_{N,1}N^{1/2}\right)^n\end{eqnarray*}
where in the third equality we use $E_Q(K(N))=N$ because the $h(n,x)$ has mean zero and second moment of simple random walk of length $N$ in dimension $d=1$ is $N$, in the fifth equality we use that the $0$-th term in the second moment expansion of $K(N)$ is $N^2$ (see Lemma \ref{3}) and the last inequality follows from Proposition \ref{d12M} ii).\end{proof}

\begin{lm}\label{14} $Z(N)\rightarrow 1$ in probability as $N\rightarrow \infty$ \end{lm}
\begin{proof} For any $\epsilon>0$, by Chebyshev's inequality and using Lemma \ref{12}, we have
$$P(|Z(N)-1|>\epsilon)\leq \frac{E_Q((Z(N)-1)^2)}{\epsilon^2}\leq \frac{\sum^N_{n=1}\left(c_1c^2_{N,1}N^{1/2}\right)^n}{\epsilon^2}$$ Let $f(N)=c_1c^2_{N,1}N^{1/2}$. By choice of $c_{N,1}$ (see (\ref{SLG})), $\lim_{N\rightarrow \infty}c^2_{N,1}N^{1/2}=0$, in particular $\lim_{N\rightarrow \infty}f(N)=0$, which says given $\delta>0$ small, there exists $K$ such that for $N\geq K$, $f(N)<\delta$ (note that $f(N)\geq 0$), but then $$S_N:=\sum^N_{n=1}f(N)^n< \sum^N_{n=1}\delta^n=\frac{\delta-\delta^{N+1}}{1-\delta}<\frac{\delta}{1-\delta}$$ Since $\delta>0$ is arbitrary and for large $N$, $S_N$ is arbitrarily small, so $S_N\rightarrow 0$ as $N\rightarrow \infty$. We conclude $Z(N)\rightarrow 1$ in probability as $N\rightarrow \infty$.
\end{proof}

\begin{lm}\label{15}$\frac{1}{Z(N)}\rightarrow 1$ in probability as $N\rightarrow \infty$\end{lm}
\begin{proof} By Lemma \ref{14}, for $\epsilon>0$ small such that $1-\epsilon>0$ and given $\delta>0$, there exists $K$ such that for $N\geq K$, $P(|Z(N)-1|\leq \epsilon)=1-P(|Z(N)-1|>\epsilon)>1-\delta$. But \begin{eqnarray*}P(|Z(N)-1|\leq \epsilon)&=&P(1-\epsilon\leq Z(N)\leq 1+\epsilon)\\&=&P\left(\frac{1}{1+\epsilon}\leq \frac{1}{Z(N)}\leq \frac{1}{1-\epsilon}\right)\\&=&P\left(1-\epsilon''\leq \frac{1}{Z(N)}\leq 1+\epsilon'\right)\\&\leq&P\left(1-\hat{\epsilon}\leq \frac{1}{Z(N)}\leq 1+\hat{\epsilon}\right)\end{eqnarray*} where first equality holds because we assume $\epsilon>0$ is small such that $1-\epsilon>0$, and $\epsilon'=\frac{1}{1-\epsilon}-1>0$, $\epsilon''=1-\frac{1}{1+\epsilon}>0$, $\hat{\epsilon}=\max\{\epsilon',\epsilon''\}$, so $1-\delta\leq P\left(1-\hat{\epsilon}\leq \frac{1}{Z(N)}\leq 1+\hat{\epsilon}\right)$ and we have $P\left(|\frac{1}{Z(N)}-1|>\epsilon\right)\rightarrow 0$ for all $\epsilon>0$ small. But we note for $\epsilon'>\epsilon$, $P\left(|\frac{1}{Z(N)}-1|>\epsilon'\right)\leq P\left(|\frac{1}{Z(N)}-1|>\epsilon\right)$, so $P\left(|\frac{1}{Z(N)}-1|>\epsilon\right)\rightarrow 0$ holds for any $\epsilon>0$. Thus $\frac{1}{Z(N)}\rightarrow 1$ in probability as $N\rightarrow \infty$.
\end{proof}

\begin{lm}\label{16}$\frac{K(N)}{N}\rightarrow 1$ in probability as $N\rightarrow \infty$\end{lm}
\begin{proof} For any $\epsilon>0$, by Chebyshev inequality and using Lemma \ref{13}, we have $$P\left(\left|\frac{K(N)}{N}-1\right|>\epsilon\right)=P(|K(N)-N|>N\epsilon)\leq \frac{E_Q((K(N)-N)^2)}{\epsilon^2N^2}\leq \frac{N^2\sum^{N}_{n=1}\left(c_1c^2_{N,1}N^{1/2}\right)^n}{\epsilon^2N^2}$$ As before since $\lim_{N\rightarrow\infty}\sum^N_{n=1}\left(c_1c^2_{N,1}N^{1/2}\right)^n=0$, we see $\frac{K(N)}{N}\rightarrow 1$ in probability as $N\rightarrow \infty$
\end{proof}

Now we are ready to show Theorem \ref{MTHM} for dimension $d=1$.

Since multiplication preserves convergence in probability, given $X_n\rightarrow X$ and $Y_n\rightarrow Y$ in probability, then $X_n\cdot Y_n\rightarrow X\cdot Y$ in probability, and recall that the mean square displacement of the polymer is $\langle \omega(N)^2\rangle_{N,h}=\frac{K(N)}{Z(N)}$, then $$\frac{\langle \omega(N)^2\rangle_{N,h}}{N}=\frac{K(N)}{N}\cdot \frac{1}{Z(N)}$$ By Lemma \ref{15}, $\frac{1}{Z(N)}\rightarrow 1$ in probability and by Lemma \ref{16}, $\frac{K(N)}{N}\rightarrow 1$ in probability, we conclude $\frac{\langle \omega(N)^2\rangle_{N,h}}{N}\rightarrow 1$ in probability.

\section{Diffusivity of Rescaled Random Polymer in $d=2$}
In this section, we are going to show Proposition \ref{d22M} and Theorem \ref{MTHM} for dimension $d=2$.

\subsection*{Section 4.1}
In this subsection, we are going to show Proposition \ref{d22M} i) in a series of lemmas.

\begin{lm}\label{17} \begin{eqnarray*}\sum_{x_1,\ldots, x_n}\prod^n_{k=1}p^2_0(i_k-i_{k-1}, x_k-x_{k-1})\leq c_2^n\prod^n_{k=1}(i_k-i_{k-1})^{-1}\end{eqnarray*} for some constant $c_2$ that depends only on the dimension $d=2$\end{lm}

\begin{proof} For $d=2$, we see (\ref{GLE}) is at most $c_2n^{-1}$ for some constant $c_2$. As in the proof of Lemma \ref{4} using this uniform estimate for the $p_0(i_k-i_{k-1},x_k-x_{k-1})$'s, we have $$\sum_{x_1,\ldots, x_n}\prod^n_{k=1}p^2_0(i_k-i_{k-1}, x_k-x_{k-1})=\sum_{x_1}p^2_0(i_1,x_1)\cdots \sum_{x_n}p^2_0(i_n-i_{n-1},x_n-x_{n-1})$$ $$\leq c_2^ni_1^{-1}\cdots (i_n-i_{n-1})^{-1}\sum_{x_1}p_0(i_1,x_1)\cdots \sum_{x_n}p_0(i_n-i_{n-1},x_n-x_{n-1})=c_2^n\prod^n_{k=1}(i_k-i_{k-1})^{-1}$$ for some constant $c_2$ that depends only on the dimension $d=2$\end{proof}

\begin{lm}\label{18} \begin{eqnarray*}\sum_{1\leq i_1< \cdots <i_n\leq N}c_2^nc^{2n}_{N,2}\prod^n_{k=1}(i_k-i_{k-1})^{-1}\leq \left(c_2c^2_{N,2} \log N\right)^n\end{eqnarray*}
\end{lm}

\begin{proof}
$$\sum_{1\leq i_1< \cdots <i_n\leq N}c_2^nc^{2n}_{N,2}\prod^n_{k=1}(i_k-i_{k-1})^{-1} $$ $$=c_2^nc^{2n}_{N,2}\sum^{N-(n-1)}_{i_1=1}\cdots \sum^{N-1}_{i_{n-1}=i_{n-2}+1}i_1^{-1}\cdots (i_{n-1}-i_{n-2})^{-1} \sum^N_{i_n=i_{n-1}+1}(i_n-i_{n-1})^{-1}$$ $$\leq c_2^nc_{N,2}^{2n}\sum^{N-(n-1)}_{i_1=1}\cdots \sum^{N-1}_{i_{n-1}=i_{n-2}+1}i_1^{-1}\cdots (i_{n-1}-i_{n-2})^{-1} 10\log N$$ Last inequality holds because $\sum^{N}_{k=1}k^{-1}\leq 1+\int^N_1x^{-1}dx=1+\log N\leq 10\log N$. Continuing from above and arguing similarly to estimate each sum in the expression we have
$$\leq c_2^nc^{2n}_{N,2}\log N\sum^{N-(n-1)}_{i_1=1}\cdots \sum^{N-1}_{i_{n-1}=i_{n-2}+1}i_1^{-1}\cdots (i_{n-1}-i_{n-2})^{-1}\leq \left(c_2c^{2}_{N,2}\log N\right)^n$$ where the constant $c_2$ will change from line to line (again it depends only on the dimension $d=2$).
\end{proof}
We conclude by Lemma \ref{2} that Proposition \ref{d22M} i) holds.

\subsection*{Section 4.2}In this subsection, we are going to show Proposition \ref{d22M} ii) in a series of lemmas.

By standard computations of the partial moments of simple random walk of length $n$ in dimension $d=2$ using characteristic function, for $x=(x_1,x_2)$, we have $$\sum_xx^2_1p_0(n,x)=\frac{n}{2};\ \ \sum_xx^4_1p_0(n,x)=\frac{3n^2-n}{4};\ \ \sum_{x}x^2_1x^2_2p_0(n,x)=\frac{n(n-1)}{4}$$ so the second and fourth moments are respectively $$\sum_x|x|^2p_0(n,x)=n;\ \ \ \sum_x|x|^4p_0(n,x)=2n^2-n$$

\begin{lm}\label{19} \begin{eqnarray*}\sum_{x}|x|^2p_0(N-i_n,x-x_n)=(N-i_n)+|x_n|^2\end{eqnarray*}\end{lm}
\begin{proof} $$\sum_x |x|^2p_0(N-i_n,x-x_n)=\sum_{x-x_n}|x|^2p_0(N-i_n,x-x_n)=\sum_x|x+x_n|^2p_0(N-i_n,x)$$ $$=\sum_x|x|^2p_0(N-i_n,x)+\sum_x|x_n|^2p_0(N-i_n,x)=(N-i_n)+|x_n|^2$$ where third equality holds because any odd partial moments of simple random walk vanish.\end{proof}

\begin{lm}\label{20}\begin{eqnarray*}\sum_{x_k}|x_k|^4p_0(i_k-i_{k-1},x_k-x_{k-1})=2(i_k-i_{k-1})^2-(i_k-i_{k-1})+|x_{k-1}|^4+4|x_{k-1}|^2(i_k-i_{k-1})
\end{eqnarray*}\end{lm}

\begin{proof}$$\sum_{x_k}|x_k|^4p_0(i_k-i_{k-1},x_k-x_{k-1}) $$ $$=\sum_{x_k-x_{k-1}}|x_k|^4p_0(i_k-i_{k-1},x_k-x_{k-1})=\sum_{x_k}|x_k+x_{k-1}|^4p_0(i_k-i_{k-1},x_k)
$$ $$=\sum_{x_k}|x_k|^4p_0(i_k-i_{k-1},x_k)+|x_{k-1}|^4\sum_{x_k}p_0(i_k-i_{k-1},x_k)+2|x_{k-1}|^2\sum_{x_k}|x_k|^2p_0(i_k-i_{k-1},x_k)
$$ $$+4x^2_{k-1,1}\sum_{x_k}x^2_{k,1}p_0(i_k-i_{k-1},x_k)+4x^2_{k-1,2}\sum_{x_k}x^2_{k,2}p_0(i_k-i_{k-1},x_k)$$
$$=2(i_k-i_{k-1})^2-(i_k-i_{k-1})+|x_{k-1}|^4+2|x_{k-1}|^2(i_k-i_{k-1})+4x^2_{k-1,1}\frac{i_k-i_{k-1}}{2}+4x^2_{k-1,2}\frac{i_k-i_{k-1}}{2}$$
$$=2(i_k-i_{k-1})^2-(i_k-i_{k-1})+|x_{k-1}|^4+4|x_{k-1}|^2(i_k-i_{k-1})$$ \end{proof}

\begin{lm}\label{21}
\begin{eqnarray*}\sum_{x_1,\ldots, x_n}\prod^n_{k=1}p_0(i_k-i_{k-1},x_k-x_{k-1})((N-i_n)^2+2(N-i_n)|x_n|^2+|x_n|^4)\end{eqnarray*}
$$ =(N-i_n)^2+2(N-i_n)i_n+2\sum^n_{k=1}(i_k-i_{k-1})^2-i_n+4\sum^n_{k=1}(i_k-i_{k-1})i_{k-1}$$\end{lm}
\begin{proof} It follows by induction on $n$ as in Lemma \ref{8}\end{proof}

\begin{lm}\label{22} \begin{eqnarray*}\sum_{x_1,\ldots, x_n} \prod^n_{k=1}p^2_0(i_k-i_{k-1}, x_k-x_{k-1})\left(\sum_x |x|^2p_0(N-i_n, x-x_n)\right)^2\leq c^n_2N^2\prod^n_{k=1}(i_k-i_{k-1})^{-1}\end{eqnarray*}\end{lm}

\begin{proof} Using the uniform estimate (\ref{GLE}) for $d=2$ on the transition probability as in the proof of Lemma \ref{17} we get $$\sum_{x_1,\ldots, x_n} \prod^n_{k=1}p^2_0(i_k-i_{k-1}, x_k-x_{k-1})\left(\sum_x |x|^2p_0(N-i_n, x-x_n)\right)^2 $$ $$\leq c_2^ni_1^{-1}(i_2-i_1)^{-1}\cdots (i_n-i_{n-1})^{-1}\sum_{x_1,\ldots, x_n} \prod^n_{k=1}p_0(i_k-i_{k-1}, x_k-x_{k-1})\left((N-i_n)+|x_n|^2\right)^2$$ $$= c_2^ni_1^{-1}(i_2-i_1)^{-1}\cdots (i_n-i_{n-1})^{-1}\sum_{x_1,\ldots, x_n} \prod^n_{k=1}p_0(i_k-i_{k-1}, x_k-x_{k-1})$$ $$\times \left((N-i_n)^2+2(N-i_n)|x_n|^2+|x_n|^4\right)$$ $$=c_2^ni_1^{-1}(i_2-i_1)^{-1}\cdots (i_n-i_{n-1})^{-1}
$$ $$\times \left((N-i_n)^2+2(N-i_n)i_n+2\sum^n_{k=1}(i_k-i_{k-1})^2-i_n+4\sum^n_{k=1}(i_k-i_{k-1})i_{k-1}\right)$$ $$\leq c_2^nN^2\prod^n_{k=1}(i_k-i_{k-1})^{-1}$$ where in the first inequality we also use Lemma \ref{19} to compute the second moment, in the second equality we use Lemma \ref{21} and and in the last inequality we use estimate similar to that in Lemma \ref{9}.
\end{proof}

\begin{lm}\label{23}\begin{eqnarray*}\sum_{1\leq i_1< \cdots <i_n\leq N}c^n_2c^{2n}_{N,2}N^2\prod^n_{k=1}(i_k-i_{k-1})^{-1}\leq N^2\left(c_2c^2_{N,2}\log N\right)^n\end{eqnarray*}\end{lm}

\begin{proof} It follows from Lemma \ref{18}.
\end{proof}

We conclude by Lemma \ref{3} that Proposition \ref{d22M} ii) holds.

\subsection*{Section 4.3}
In this subsection, we are going to show Theorem \ref{MTHM} for $d=2$.

Clearly, Lemmas \ref{12} and \ref{13} hold for dimension $d=2$ with $c_2c^2_{N,2}\log N$ instead of $c_1c^2_{N,1}N^{1/2}$, precisely we have $$
E_Q\left((Z(N)-1)^2\right)\leq \sum^N_{n=1}\left(c_2c^2_{N,2}\log N\right)^n;\ \ E_Q\left((K(N)-N)^2\right)\leq N^2\sum^N_{n=1}\left(c_2c^2_{N,2}\log N\right)^n$$

By choice of $c_{N,2}$ (see (\ref{SLG})), as in the proof of Lemma \ref{14} it implies $$\lim_{N\rightarrow \infty}\sum^N_{n=1}\left(c_2c^2_{N,2}\log N\right)^n=0$$ thus we see Lemmas \ref{14}, \ref{15} and \ref{16} hold with suitable changes. We conclude that for dimension $d=2$, $\frac{\langle \omega(N)^2\rangle_{N,h}}{N}\rightarrow 1$ in probability.

\section{Other Results}
In this section, we are going to show other results in the diffusive regime.

\begin{te}\label{28}  With rescaling of the polymer density by $c_{N,d}$ for $d=1,2$, there exists normalizing constants $a_{N,d}$ such that $$a_{N,d}\left(Z_N-1\right)\Rightarrow \xi$$ where $\xi$ is some Gaussian random variable. \end{te}

First, we have the following lemma.

\begin{lm}\label{25} For $x_1,\ldots, x_n\in \mathbb{Z}^d$, \begin{eqnarray*}\sum_{x_1,\ldots, x_n}\prod^n_{k=1}p^2_0(i_k-i_{k-1}, x_k-x_{k-1})=\prod^n_{k=1}p_0(2(i_k-i_{k-1}),0)\end{eqnarray*}\end{lm}

\begin{proof} Note that for transition probability $p_0(n,x)$ of the simple random walk starting at $0$ and ending at $x$ at time $n$, by spatial homogeneity it is same as the transition probability $p_x(n,2x)$ of the simple random walk starting at $x$ and ending at $2x$ at time $n$. Furthermore, by reflecting each step walk takes to reach from $x$ to $2x$, for example in dimension $d=2$ if original walk goes up, then the reflecting walk goes down and if original walk goes right, then the reflecting walk goes left, we get a reflecting walk starting at $x$ and ending at $0$ at time $n$ with transition probability $p_x(n,0)$ such that $$p_0(n,x)=p_x(n,2x)=p_x(n,0)$$ Using Chapman-Kolmogorov equality for the simple random walk, we have $$\sum_xp^2_0(n,x)=\sum_xp_0(n,x)p_0(n,x)=\sum_xp_0(n,x)p_x(n,0)=p_0(2n,0)$$ We conclude that $$\sum_{x_1,\ldots, x_n}\prod^n_{k=1}p^2_0(i_k-i_{k-1}, x_k-x_{k-1}) $$ $$=\sum_{x_1}p^2_0(i_1,x_1)\cdots\sum_{x_{n-1}}p^2_0(i_{n-1}-i_{n-2},x_{n-1}-x_{n-2})
\sum_{x_n}p^2_0(i_n-i_{n-1},x_n-x_{n-1})$$ $$=\sum_{x_1}p^2_0(i_1,x_1)\cdots\sum_{x_{n-1}}p^2_0(i_{n-1}-i_{n-2},x_{n-1}-x_{n-2})\sum_{x_n-x_{n-1}}p^2_0(i_n-i_{n-1},x_n-x_{n-1})$$
$$=\sum_{x_1}p^2_0(i_1,x_1)\cdots\sum_{x_{n-1}}p^2_0(i_{n-1}-i_{n-2},x_{n-1}-x_{n-2})\sum_{x_n}p^2_0(i_n-i_{n-1},x_n)$$ $$=\sum_{x_1}p^2_0(i_1,x_1)\cdots\sum_{x_{n-1}}p^2_0(i_{n-1}-i_{n-2},x_{n-1}-x_{n-2})p_0(2(i_n-i_{n-1}),0)$$
$$=\prod^n_{k=1}p_0(2(i_k-i_{k-1}),0)$$
\end{proof}

To show Theorem \ref{28}, we write the partition function as a sum of two parts

\begin{lm}\label{29} $Z_N-1=\sum^N_{k=1}f_k+R_N$ \end{lm}

\begin{proof} \begin{eqnarray*}Z_N-1&=&g_1+\sum^N_{n=2}g_n\\&=&
\int \sum^N_{k=1}c_{N,d}h(k,\omega(k))dP^N_0(\omega)+\sum^N_{n=2}g_n
\\&=& \sum^N_{k=1}\sum_xc_{N,d}h(k,x)p_0(k,x)+R_N\\&=&\sum^N_{k=1}f_k+R_N\end{eqnarray*}
where the $g_n$ are as in Lemma \ref{2}
\end{proof}

By the following two propositions, Theorem \ref{28} follows since if $X_n\Rightarrow X$, $Y_n\Rightarrow a$ where $a$ is a constant, then $X_n+Y_n\Rightarrow X+a$.

\begin{pr}\label{30} $\sum^N_{k=1}a_{N,d}f_k\Rightarrow \xi$ where $\xi$ is some Gaussian random variable.\end{pr}

\begin{proof} By definition, $f_k=c_{N,d}\sum_x h(k,x)p_0(k,x)$, if we let $X_{k,N}=a_{N,d}f_k$, then to show proposition it suffices to check conditions in the Lindeberg-Feller Theorem are satisfied, i.e. $\sum^N_{k=1} E_QX^2_{k,N}\rightarrow c$ where $c>0$, and for all $\epsilon>0$, $\lim_N\sum^N_{k=1}E_Q\left(X^2_{k,N};\ |X_{k,N}|>\epsilon\right)=0$ By direct computations, we find \begin{eqnarray*}E_QX^2_{k,N}&=&a^2_{N,d}c^2_{N,d}
E_Q\left(\sum_x h^2(k,x)p^2_0(k,x)+\sum_{x\neq x'}h(k,x)h(k,x')p_0(k,x)p_0(k,x')\right)\\&=&a^2_{N,d}c^2_{N,d}\sum_xp^2_0(k,x)\\&=&a^2_{N,d}c^2_{N,d}p_0(2k,0)\end{eqnarray*} where last equality follows from Lemma \ref{25}. Using estimate of the transition probability $p_0(n,x)$ as in \ref{GLE}, in $d=1$, $p_0(2k,0)=\pi^{-1/2}k^{-1/2}+r_{2k}(0)$, where $|r_{2k}(0)|\leq c_1k^{-3/2}$ and in $d=2$, $p_0(2k,0)=\pi^{-1}k^{-1}+r_{2k}(0)$, where $|r_{2k}(0)|\leq c_2k^{-2}$, we see in $d=2$, $\sum^N_{k=1}E_QX^2_{k,N}=a^2_{N,2}c^2_{N,2}\sum^N_{k=1}p_0(2k,0)=a^2_{N,2}c^2_{N,2}\left(\sum^N_{k=1}\pi^{-1}k^{-1}+r_{2k}(0)\right)$. If we take $a_{N,2}=\left(c^2_{N,2}\log N\right)^{-1/2}$, then $$a^2_{N,2}c^2_{N,2}\left(\sum^N_{k=1}\pi^{-1}k^{-1}+r_{2k}(0)\right)\rightarrow \pi^{-1}$$
(because $1-(N+1)^{-1}\leq \sum^N_{k=1}k^{-2}\leq 2-N^{-1}$, so $a^2_{N,2}c^2_{N,2}\sum^{N}_{k=1}k^{-2}\rightarrow 0$).

Next, for given $\epsilon>0$, we also find \begin{eqnarray*} E_Q\left(X^2_{k,N};\ |X_{k,N}|\geq \epsilon\right)&=&E_Q\left(a^2_{N,2}f^2_k;\ a_{N,2}|f_k|>\epsilon\right)\\&\leq & a^2_{N,2}c^2_{N,2}E_Q\left(\sum_xp^2_0(k,x)1_{a_{N,2}|f_k|>\epsilon}\right)\\&+&a^2_{N,2}c^2_{N,2}E_Q\left(
\sum_{x\neq x'}|h(k,x)h(k,x')p_0(k,x)p_0(k,x')|1_{a_{N,2}|f_k|>\epsilon}\right)\\&=&
a^2_{N,2}c^2_{N,2}Q(a_{N,2}|f_k|>\epsilon)\left(p_0(2k,0)+\sum_{x\neq x'}p_0(k,x)p_0(k,x')\right)
\end{eqnarray*}
But $Q(a_{N,2}|f_k|>\epsilon)=Q(|\sum_xh(k,x)p_0(k,x)|>a^{-1}_{N,2}c^{-1}_{N,2}\epsilon)\leq Q(\sum_x|h(k,x)|p_0(k,x)>a^{-1}_{N,2}c^{-1}_{N,2}\epsilon)=Q(1>a^{-1}_{N,2}c^{-1}_{N,2}\epsilon)=Q(1>(\log N)^{1/2}\epsilon)=0$ for $N$ large. Thus $$\lim_N\sum^N_{k=1}E_Q\left(X^2_{k,N};\ |X_{k,N}|>\epsilon\right)=0$$

Similarly, we can check conditions in the Lindeberg-Feller theorem are satisfied for $d=1$ if we take $a_{N,1}=(c^2_{N,1}N^{1/2})^{-1/2}$.
\end{proof}

\begin{pr}\label{31} $a_{N,d}R_N\Rightarrow 0$ \end{pr}

\begin{proof}Note that $a_{N,d}R_N\Rightarrow 0$ if and only if $a_{N,d}R_N\rightarrow 0$ in probability and $a_{N,d}R_N\rightarrow 0$ in probability if $E_Q(a_{N,d}R_N)^2\rightarrow 0$ by Chebyshev inequality.

We show the proposition for $d=2$, and it is similar for $d=1$.

By definition, $R_N=\sum^N_{n=2}g_n$, so $E_Q(a_{N,d}R_N)^2=a^2_{N,d}E_Q\left(\sum^N_{n=2}g^2_n+\sum_{n\neq m} g_ng_m\right)$. From Lemma \ref{2}, we know that for $n\neq m$, $E_Qg_ng_m=0$, and $E_Qg^2_n=\sum_{1\leq i_1<\cdots<i_n\leq N}c^{2n}_{N,2}\prod^n_{k=1}p_0(2(i_k-i_{k-1}),0)\leq \left(c_2c^2_{N,2}\log N\right)^n$. Recall $a_{N,2}=\left(c^2_{N,2}\log N\right)^{-1/2}$ for $d=2$ so $$a^2_{N,2}E_Q\sum^N_{n=2}g^2_n\leq\sum^N_{n=2}
a^2_{N,2}(c_2c^2_{N,2}\log N)^n=\sum^N_{n=2}c_2(c_2c^2_{N,2}\log N)^{n-1}\rightarrow 0$$
\end{proof}

\end{document}